\documentclass[a4paper, 11pt]{amsart}

\usepackage{a4wide,amsmath, amsfonts, amssymb, amsthm, graphicx, enumerate}
\usepackage[colorlinks]{hyperref}
\hypersetup{linkcolor=blue, urlcolor=blue, citecolor=red}
\numberwithin{equation}{section}

\let\oldmarginpar\marginpar
\renewcommand\marginpar[1]{\-\oldmarginpar[\raggedleft\footnotesize #1]%
{\raggedright\footnotesize #1}}

\newcounter{ccounter}

\newcounter{ccountertwo}

\newtheorem{thm}{Theorem}[section]
\newtheorem{cor}[thm]{Corollary}
\newtheorem{lem}[thm]{Lemma}
\newtheorem{prop}[thm]{Proposition}

{\theoremstyle{definition}
\newtheorem{claim}[thm]{Claim}
\newtheorem{defn}[thm]{Definition}

\newtheorem{question}[thm]{Question}}
\newcommand{\aut}{\operatorname{Aut}}

\newcommand{\dom}{\operatorname{dom}}
\newcommand{\ran}{\operatorname{ran}}

\newcommand{\N}{\mathbb{N}}
\newcommand{\Z}{\mathbb{Z}}
\newcommand{\Q}{\mathbb{Q}}
\newcommand{\R}{\mathbb{R}}

\newcommand{\sym}{\operatorname{Sym}}

\renewcommand{\and}{\text{ and }}

\newcommand{\set}[2]{\{#1:#2\}}
\renewcommand{\to}{\longrightarrow}

\author{J. Hyde} 
\author{J. Jonu\v sas} 
\author{J. D. Mitchell}
\author{Y. H. P\'eresse}
\title[Universal sequences]{Sets of universal sequences for the symmetric group
and analogous semigroups}

\begin{document}
\maketitle

\begin{abstract}
  A \textit{universal sequence} for a group or semigroup $S$ is a sequence of
  words $w_1, w_2, \ldots$ such that for any sequence $s_1, s_2, \ldots\in S$,
  the equations $w_n = s_n$, $n\in \mathbb{N}$, can be solved simultaneously in
  $S$. For example, Galvin showed that the sequence
  $(a^{-1}(a^nba^{-n})b^{-1}(a^nb^{-1}a^{-n})ba)_{n\in\mathbb{N}}$ is universal
  for the symmetric group $\sym(X)$ when $X$ is infinite, and Sierpi\'nski showed
  that $(a ^ 2 b ^ 3 (abab ^ 3) ^ {n + 1} ab ^ 2 ab ^ 3)_{n\in \N}$ is
  universal for the monoid $X ^ X$ of functions from the infinite set $X$ to
  itself.

  In this paper, we show that under some conditions, the set of universal
  sequences for the symmetric group on an infinite set $X$ is independent of
  the cardinality of $X$.  More precisely,  we show that if $Y$ is any set such
  that $|Y| \geq |X|$, then every universal sequence for $\sym(X)$ is also
  universal for $\sym(Y)$.  If $|X| > 2 ^ {\aleph_0}$, then the converse also
  holds. It is shown that an analogue of this theorem holds in the context of
  inverse semigroups, where the role of the symmetric group is played by the
  symmetric inverse monoid.  In the general context of semigroups, the full
  transformation monoid $X ^ X$ is the natural analogue of
  the symmetric group and the symmetric inverse monoid. If $X$ and $Y$ are
  arbitrary infinite sets, then it is an open question as to whether or not
  every sequence that is universal for $X ^ X$ is also universal for $Y ^ Y$.
  However, we obtain a sufficient condition for a sequence to be universal for
  $X ^ X$ which does not depend on the cardinality of $X$. A large class of
  sequences satisfy this condition, and hence are universal for $X ^ X$ for
  every infinite set $X$. 
\end{abstract}


\section{Introduction}

  Let $F$ be a free group, let $w\in F$, and let $G$ be a group. We say that
  the word $w$ is \textit{group universal} for $G$ if for all $g\in G$ there
  exists a group homomorphism $\phi: F \to G$ such that $(w)\phi = g$. For
  example, Or\'e~\cite{Ore1951aa} showed that every element of the symmetric
  group $\sym(X)$ on an infinite set $X$ is a commutator, that is,
  $x^{-1}y^{-1}xy$ is a universal word for $\sym(X)$ when $X$ is infinite. More
  generally, every element is a commutator in any Polish group with a comeagre
  conjugacy class~\cite{Kechris2007aa}. There are many such groups in addition
  to the symmetric group; for example, the automorphism group of the countable
  random graph; see~\cite{Kechris2007aa} for further examples. 

  Something much stronger than \'Ore's Theorem holds for the symmetric group:
  any word $w$, which is not a proper power of another word, in any free group
  $F$ is group universal for $\sym(X)$. 
  Silberger~\cite{Silberger1983aa}, Droste~\cite{Droste1985ac}, and
  Mycielski~\cite{Mycielski1987aa} proved some special cases of this theorem,
  the proof of which was completed by Lyndon~\cite{Lyndon1990aa} and Dougherty
  and Mycielski~\cite{Dougherty1999aa}.  Droste and Truss~\cite{Droste2006aa}
  proved that certain classes of words are group universal for the automorphism
  group of the countably infinite random graph.  

  Roughly speaking, if $w$ is a group universal word for $G$, then the equation
  $w = g$ can be solved for all $g\in G$. It is natural to extend this to
  solving simultaneous equations. If $F$ is a free group and $w_1, w_2,
  \ldots\in F$, then given any sequence $g_1, g_2, \ldots\in G$, is it possible
  to find a homomorphism $\phi:F \to G$ such that $(w_i)\phi = g_i$ for all
  $i\in \N$? The sequence $w_1, w_2, \ldots\in F$ is \textit{group universal}
  for $G$ if such a homomorphism exists for all $g_1, g_2, \ldots\in G$.
  In~\cite{Galvin1995aa}, Galvin showed that
  $(a^{-1}(a^nba^{-n})b^{-1}(a^nb^{-1}a^{-n})ba)_{n\in\mathbb{N}}$ is universal
  for the symmetric group on an infinite set. 
  Truss~\cite{Truss2009aa} showed
  that Galvin's proof works essentially unchanged for the groups of
  homeomorphisms of the Cantor space, the rationals $\Q$, and the irrationals
  $\R \setminus \Q$.  
  In~\cite{Hyde2016aa}, the present authors showed that there is an 
  2-letter universal sequence for the group $\aut(\Q, \leq)$ of
  order-automorphisms of the rationals $\Q$.
  In~\cite{Droste1987aa}, Droste and Shelah consider a more general notion of
  universality than that defined here. As a special case, it follows from the
  result in~\cite{Droste1987aa} that if $X$ and $Y$ are sets such that $|X|,
  |Y| > 2^{\aleph_0}$, then a finite sequence is universal, in our sense, for
  $\sym(X)$ if and only if it is universal for $\sym(Y)$. In
  Corollary~\ref{cor-universal-cardinlity-sym}, we extend this result to
  infinite universal sequences. 

  Let $A$ be a finite set, called an \textit{alphabet}, and let $A ^ +$ denote
  the \textit{free semigroup} consisting of all of the non-empty words over $A$
  with multiplication being simply the concatenation of words. 

  \begin{defn}\label{de-universal-sequence}
    Let $S$ be a semigroup and let $A$ be any alphabet. Then an infinite 
    sequence of words $w_1, w_2,\ldots\in A^{+}$ is \emph{semigroup
    universal} for $S$ if for any sequence $s_1, s_2, \ldots \in S$ there
    exists a homomorphism $\phi:A^+\to S$ such that $(w_n)\phi=s_n$ for all
    $n \geq 1$. 
  \end{defn}

  Suppose that $G$ is a group.  Since the free semigroup on a finite alphabet
  $A$ is a subsemigroup of the free group on $A$, it follows that every
  semigroup universal sequence for $G$ is also a group universal sequence for
  $G$. On the other hand, every group universal sequence over $A$ for $G$ is a
  semigroup universal sequence for $G$ over $A\cup A ^ {-1}$. So, broadly
  speaking, the notion of semigroup universal sequences includes the
  corresponding notion for groups, and as such we will restrict ourselves to
  considering only semigroup universal sequences. 

  The existence of a universal sequence over a finite alphabet for a semigroup
  $S$ implies that $S$ has several further properties.  For instance, if $S$ is
  such a semigroup and $X$ is any generating set for $S$, then there exists an
  $n\in \N$ such that every element of $S$ can be given as a product over $X$
  of length at most $n$. This is known as the \textit{Bergman property} after
  Bergman's seminal paper~\cite{Bergman2006aa}; see also~\cite{Maltcev2009aa,
  Mitchell2011ab}.  A group $G$ with the Bergman property automatically
  satisfies Serr\'e's properties (FA) and (FH); see~\cite{Kechris2007aa}. There
  are, of course, many groups which have no universal sequences. For example,
  since every group with a universal sequence has property (FA), any group with
  $\Z$ as a homomorphic image has no universal sequences.

  The question of whether a universal sequence exists for a given semigroup
  has a long history, which predates \'Ore's Theorem~\cite{Ore1951aa}. 
  In 1934, Sierpi\'nski~\cite{Sierpinski1934aa} showed that $(a b ^ {n - 1} c d
  ^ {n - 1})_{n \in \N}$ is a universal sequence for the semigroup of continuous
  functions on the closed unit interval $[0, 1]$ in $\mathbb{R}$, and in 1935,
  \cite{Sierpinski1935aa} showed that $(a ^ 2 b ^ 3 (abab ^ 3) ^
  {n + 1} ab ^ 2 ab ^ 3)_{n\in \N}$ is universal for the semigroup $X ^ X$ of
  functions from the infinite set $X$ to itself where the operation is
  composition of functions.  Several further universal sequences are known for
  $X ^ X$ when $X$ is infinite, such as $(aba^{n+1}b^2)_{n\in\N}$;
  see Banach~\cite{Banach1935aa}.
  It can be shown that universal sequences are preserved by homomorphisms of
  semigroups, and a, more or less straightforward, counting argument shows that
  every semigroup with a universal sequence is of cardinality at least continuum.
  It follows that a semigroup with a countable homomorphic image has no
  universal sequences.  Some more recent results about universal sequences of
  semigroups include~\cite[Theorem
  31]{East2012aa},~\cite[Theorem 37]{East2014aa}, and~\cite[Theorem
  6.1]{East2017aa}.  See~\cite{Mitchell2011ab} and the references therein for
    further background on universal sequences for semigroups. 

  Given that a universal sequence for a given semigroup $S$ exists, it is
  natural to attempt to classify all of the universal sequences for $S$. For
  instance, given that universal words for the symmetric group $\sym(X)$ on any
  infinite set $X$ are completely classified, we might ask for a classification
  of universal sequences for $\sym(X)$. We do not provide such a
  classification, but in Section~\ref{sect-symm}, we show that if $X$ is any
  infinite set and $Y$ is any set containing $X$, then every sequence that is
  universal for the symmetric group $\sym(X)$ on $X$ is universal for
  $\sym(Y)$. The converse holds when $|X|$ is greater than $2 ^ {\aleph_0}$. It
  is, however, not known whether it remains true if $|X| \leq 2^{\aleph_0}$,
  see Question~\ref{question:remove-ass}. We also show that the analogous
  results hold for the symmetric inverse monoids.
  
  In the context of clones of polymorphisms, the natural equivalent of words
  are \textit{terms}.  In \cite{McNulty:1976mi}, McNulty gave a sufficient
  condition for such a sequence of terms to be universal. A special case of our
  main result in Section~\ref{sect-trans} and of McNulty's result, is
  Corollary~\ref{cor-no-pre-suf-overlap}.  Taylor~\cite{Taylor1981aa} showed
  that the question of whether or not a term is universal for the clone of
  polymorphisms is undecidable. 

  The question of describing universal words for $X ^ X$, and whether or not
  such words depend on the cardinality of $X$, is Problem 27 in~\cite{:2006fu}.
  As a partial result in the direction of solving this problem in
  Section~\ref{sect-trans}, we give a natural sufficient condition under which
  a sequence over a 2-letter alphabet is universal for $X ^ X$.  A special case
  of this condition is any sequence of distinct words $w_1, w_2, \ldots$ where
  no $w_i$ is a subword of any $w_j$, $i\not=j$, and no proper prefix of any
  $w_i$ is a suffix of any $w_j$. We will show in the next proposition that the
  apparent restriction to $2$-letter alphabets is, in fact, not a restriction
  at all.  

  Throughout the paper we use the convention that a countable set can be finite
  or infinite.

\begin{prop}[cf. Problem 27 in \cite{:2006fu}] 
  \label{prop-min-us-rank}
  Let $S$ be a semigroup and let $A$ be an alphabet such that there is a
  universal sequence for $S$ over $A$. Then for every countable
  alphabet $B$ there exists a function $\phi: (B^+)^\N \to (A^+)^\N$ such that
  $(w_1, w_2, \ldots)\in (B^+) ^ \N$ is universal for $S$ if and only if
  $(w_1, w_2, \ldots)\phi\in (A^+)^\N$ is universal for $S$.
\end{prop}

\begin{proof} 
  By assumption, there exists a universal sequence $(w_1, w_2,\ldots)\in (A ^
  +) ^ \N$ for $S$.  If $(u_1, u_2 \ldots)$ is a sequence over $B = \{b_1,
  b_2, \ldots\}$, then for every $m\in \N$ we define $v_m \in (A ^ +) ^
  \N$ to be the word obtained by replacing every occurrence of every letter
  $b_j$ in $u_m \in B ^ +$ by the word $w_j\in A ^ +$.  We define $\phi$ by
  $(u_1, u_2, \ldots)\phi = (v_1, v_2,\ldots)$.

  If $(u_1, u_2, \ldots)$ is universal for $S$ over $B$, then for any choice of
  $s_1, s_2, \ldots \in S$ there is a homomorphism $\Phi: B^+ \to S$ such that
  $(u_i)\Phi = s_i$ for all $i$.  Since $(w_1, w_2,\ldots)$ is universal there
  is a homomorphism $\Psi: A^+ \to S$ such that $(w_j)\Psi = (b_j)\Phi$ for all
  $j\in \{1, \ldots, n\}$. Then $(v_i)\Psi = (u_i)\Phi = s_i$ for all $i$, and
  so $(v_1, v_2, \ldots)$ is universal also.

  On the other hand, if $(v_1, v_2, \ldots)$ is universal, then for every
  choice of $s_1, s_2, \ldots \in S$ there is a homomorphism $\Phi : A^+ \to S$
  such that $(v_i)\Phi = s_i$ for all $i$. If $\Psi :B^+ \to S$ is the natural
  homomorphism extending $(b_j)\Psi = (w_j)\Phi$ for all $j$, then $(u_i)\Psi =
  (v_i)\Phi = s_i$ for all $i$, and thus $(u_1, u_2, \ldots)$ is universal.
\end{proof}
 

  
  We conclude this section with some standard definitions and notation.
  A \textit{monoid} is a semigroup $M$ with an identity, that is an element
  $1_M \in M$ such that $1_M m = m 1_M = m$ for all $m \in M$.  A
  \textit{submonoid} of a monoid $M$ is a subsemigroup containing the identity
  $1_M$ of $M$.  Any semigroup can be made into a monoid by adjoining an
  identity as follows.  If $S$ is a semigroup and $1_S \notin S$, define an
  operation on $S^1 = S \cup \{1_S\}$ which extends the operation of $S$ by
  $s1_S = 1_Ss= s$ for all $s \in S^1$.  The set $S^1$ with this operation is a
  monoid. An element $0_S$ of a semigroup $S$ is called a \textit{zero} if
  $0_Ss = s0_S = 0_S$ for all $s\in S$. A zero can be adjoined to a semigroup
  $S$ in much the same way as an identity; we denote this by $S ^ 0$.  
  The \textit{free monoid} $A^*$ is obtained from $A^+$ by adjoining an
  identity $\varepsilon$, usually referred to as the \textit{empty word}.  If
  $w = a_1 \cdots a_n \in A^*$ and $i,j \in \{1, \ldots, n\}$ are
  such that $i \leq j$, then $a_1 \cdots a_{i - 1}$ is a \textit{prefix} of
  $w$, $a_{j+1} \cdots a_n$ is a \textit{suffix} of $w$, and $a_i \cdots a_j$
  is a \textit{subword} of $w$. The empty word $\varepsilon$ is a
  prefix and a suffix of every word. 

  The analogue of the symmetric group in the context of semigroups is the
  \textit{full transformation monoid} $X ^ X$ consisting of all functions from
  the set $X$ to $X$ under composition of functions. Every semigroup
  is isomorphic to a subsemigroup of some full transformation monoid; see
  \cite[Theorem 1.1.2]{Howie1995aa}.
  
  An \textit{inverse semigroup} is a semigroup $S$ such that for all $x\in S$
  there exists a unique $x ^ {-1}\in S$ such that $xx ^{-1}x = x$ and $x^ {-1}x
  x ^ {-1} = x ^ {-1}$.  A \textit{partial permutation} on a set $X$ is a
  bijection $f: A \to B$ between subsets $A$ and $B$ of $X$. The set $A$ is
  the \textit{domain} of $f$ and is denoted $\dom(f)$; the set $B$ is
  called the \textit{range} and is denoted $\ran(f)$.  If $f : X \to Y$ is a
  partial permutation and $Z \subseteq X$, then the \textit{restriction} of $f$
  to $Z$ is the partial permutation $f|_Z : Z \to Y'$ where $Y' = \{(z)f : z
  \in Z\}$ defined by $(z)f|_Z = (z)f$ for all $z \in Z$.  Under the usual
  composition of binary relations, the set $I(X)$ of all partial permutations
  on $X$ is an inverse semigroup; $I(X)$ will be referred to as the
  \textit{symmetric inverse monoid} on $X$. The Wagner-Preston Representation
  Theorem~\cite[Theorem 5.1.7]{Howie1995aa} states that every inverse semigroup
  is isomorphic to an inverse subsemigroup of $I(X)$ for some set $X$.  It is
  possible to define the notion of an inverse semigroup universal sequence,
  which is analogous to the notions for groups and semigroups. We have already
  argued that group and semigroup universal sequences are interchangeable, and
  a similar argument applies to inverse semigroups.  For the sake of brevity we
  refer to semigroup universal sequences as \textit{universal sequences}.

\section{The role of $|X|$ for universal sequences in $\sym(X)$ and $I(X)$}
\label{sect-symm}

In this section, we consider a class of semigroups which includes the symmetric
groups and symmetric inverse monoids on arbitrary infinite sets. In particular,
let $\alpha$ be either an arbitrary infinite cardinal or $0$, and let $X$ be
any set. Then we denote by $I(X, \alpha)$ the inverse subsemigroup of $I(X)$
consisting of all the partial permutations $f$ of $X$ such that $|X\setminus
\dom(f)|, |X\setminus \ran(f)| \leq \alpha$. Note that $I(X, 0) = \sym(X)$, the
symmetric group on $X$, and that $I(X, \alpha)$ is the whole of $I(X)$ for any
$\alpha \geq |X|$. Recall that an infinite cardinal $\lambda$ is \emph{regular}
if it cannot be expressed as the union of strictly less than $\lambda$ many
sets each of cardinality strictly less than $\lambda$.


The main theorem of this section is the following.

\begin{thm}\label{thm-different-cardinalities}
  Let $X$ and $Y$ be sets, and let $\alpha$ be any infinite cardinal number or
  $0$.  Then the following hold:
  \begin{enumerate}[\rm (i)]
    \item 
      if $\aleph_0\leq |X| < |Y|$ and $\alpha \in \{0, |Y|\}$, then every
      sequence that is universal for $I(X,\alpha)$ is also
      universal for $I(Y, \alpha)$;
    
    \item 
      if $2^{\aleph_0}< |X| < |Y|$, $\alpha < |X|$ or $\alpha \geq |Y|$, and $|X|$
      is a regular cardinal, then every sequence that is universal 
      for $I(Y, \alpha)$ is also universal for $I(X, \alpha)$.
  \end{enumerate}
\end{thm}
\begin{proof}
  \textbf{(i).} 
  Let $w_1, w_2, \ldots$ be a universal sequence for $I(X, \alpha)$ over some
  countable alphabet $A$, and let $s_1, s_2, \ldots\in I(Y, \alpha)$ be
  arbitrary. 
  It follows from~\cite[Proposition~2.1(ii)]{Hyde2016aa}, $w_1,
  w_2, \ldots$ is also universal for $I(X, \alpha)^{|Y|}$.
  
  We define $S$ to be the inverse semigroup generated by $\{s_1, s_2,
  \ldots\}$. Then $S$ is countable and the sets $\set{(z)s}{s\in S ^ 1}$, where
  $z\in Y$, partition $Y$ into $|Y|$ many countable sets. We refer to these sets as
  the \textit{blocks} of $S$ on $Y$.
  Define a partition $\set{X_y}{y\in Y}$ of $Y$ such
  that each $X_y$ is a union of blocks and $|X_y| = |X|$, this is possible
  since the blocks are countable and $X$ is infinite. For every $y\in Y$, let
  $\mu_y: X_y\to X$ be any bijection. It follows that $f: S \to
  I(X,\alpha)^{|Y|}$ defined by $(s)f = (\mu_y^{-1}s\mu_y)_{y\in Y}$ is an
  injective homomorphism. 

  Define a map $g: I(X, \alpha) ^ {|Y|} \to I(Y)$ by 
  \[
    ((b_y)_{y\in Y})g = \bigcup_{y\in Y} \mu_yb_y\mu_{y}^{-1}.
  \]
  Since the sets $X_y$ partition $Y$, $((b_y)_{y\in Y})g$ is a well-defined
  partial permutation of $Y$.  We will show that if $\alpha$ is either $|Y|$ or
  $0$, then, in fact, $g$ is contained in $I(Y, \alpha)$. If $\alpha$ is $|Y|$,
  then $I(Y, \alpha) = I(Y)$, as required. Suppose that $\alpha = 0$. Then for
  every $(b_y)_{y \in Y} \in I(X, \alpha)^{|Y|}$ and every $y \in Y$
  \[
    |X_y\setminus \dom(\mu_yb_y\mu_{y}^{-1})| = |X\setminus
    \dom(b_y)| = 0 
  \]
  and similarly
  \[
    |X_y\setminus \ran(\mu_yb_y\mu_{y}^{-1})| = |X\setminus \ran(b_y)| = 0.
  \]
  Hence  the domain and range of $((b_y)_{y\in Y})g$ are both $Y$, and so
  $((b_y)_{y\in Y})g\in I(Y, \alpha)$. Hence $g: I(X, \alpha)^{|Y|} \to I(Y,
  \alpha)$ is a homomorphism, and $(s)f g = s$ for all $s \in S$.
  
  Since $w_1, w_2, \ldots$ is a universal sequence for $I(X, \alpha)^{|Y|}$,
  there exists a homomorphism $\phi: A^+ \to I(X, \alpha)^{|Y|}$ such that
  $(w_n)\phi = (s_n)f$ for all $n$, and so $\phi \circ g : A ^ + \to I(Y,
  \alpha)$ is a homomorphism and $(w_n)\phi \circ g = (s_n)fg = s_n$, as
  required. 

  \textbf{(ii).} 
  Let $w_1, w_2, \ldots$ be a universal sequence for $I(Y, \alpha)$ over some
  countable alphabet $A$, and let $s_1, s_2, \ldots\in I(X, \alpha)$ be
  arbitrary. 
  
  As in part (i) we denote the inverse subsemigroup of $I(X, \alpha)$ generated
  by $\{s_1, s_2, \ldots\}$ by $S$, and let $\Omega$ be the set of blocks of
  $S$ on $X$. We define an equivalence relation $\sim$ on $\Omega$ as follows:
  for $U, V\in \Omega$ we write $U \sim V$ if there is a
  bijection $\phi:U \to V$ such that $s_n \circ \phi = \phi \circ s_n$
  for all $n\in \N$.  In other words, $U\sim V$ if and only if the
  inverse semigroup $S$ has the same action on $U$
  and $V$, up to relabelling the points. 

  If $U\in \Omega$, then $|U| \leq \aleph_0$ and since $|X| >
  \aleph_0$, it follows that $|\Omega| = |X|$. Since a countable semigroup has
  at most $\aleph_0 ^ {\aleph_0} = 2 ^ {\aleph_0}$ distinct (partial) actions
  on a given countable set, it follows that there are at most $2 ^ {\aleph_0}$
  equivalence classes of $\sim$. Since $|\Omega| = |X| > 2 ^ {\aleph_0}$ and
  $|X|$ is a regular cardinal, $\Omega$ cannot be written as a union of $2 ^
  {\aleph_0}$ sets of cardinality strictly less than $|X|$. Hence there exists
  an equivalence class $E$ of $\sim$ such that $|E| = |X|$. 

  For a fixed $U\in E$, we define $Y'$ to be the disjoint union
  of $Y\times U$ and $X$ and also for each $n$ we define $t_n:Y'\to Y'$ by
  \begin{equation*}
    (x)t_n = 
    \begin{cases}
      (x)s_n     & x \in X \\
      (y, (z)s_n) & x = (y, z) \in Y\times U.
    \end{cases}
  \end{equation*}
  Obviously $t_n$ is a partial permutation, and 
  we will show that $t_n\in I(Y', \alpha)$. There are two cases to consider,
  when $\alpha = |Y|$ and when $\alpha < |X|$. If $\alpha = |Y|$, then $I(Y',
  \alpha)$ consists of all partial permutations on $Y'$, and
  so $t_n \in  I(Y', \alpha)$. 
  The other case is significantly more complicated.
 
  \begin{claim}
    If $\alpha < |X|$, then $t_n\in I(Y', \alpha)$ for all $n\in \N$.
  \end{claim}
  \begin{proof}
    We define
    \[
      Z = \bigcup_{m \geq 1 } \big(X\setminus \dom(s_m)\big) \cup
      \big(X\setminus \ran(s_m)\big).
    \]
    Since $Z$ is a countable union of sets with cardinality at most $\alpha$,
    $|Z| \leq \alpha$. 
 
    If $V, W\in E$ and $V\cap Z\not=\varnothing$, then we will
    show that $W\cap Z\not=\varnothing$ also. Since $V, W\in E$, there exists a
    bijection $\phi:V\to W$ such that $s_n\phi= \phi s_n$ for all $n\in \N$. 
    Suppose that $x\in V\cap Z$. Then by the definition of $Z$ there exists
    $m\in\N$ such that $x\not\in\dom(s_m)$ or $x\not\in \ran(s_m)$. 
    If $x\not\in\dom(s_m)$, then $x\not\in \dom(s_m\phi) = \dom(\phi s_m)$. But
    $x\in \dom(\phi) = V$, and so $(x)\phi \not\in \dom(s_m)$. In other words,
    $(x)\phi \in W\cap Z$, which is consequently non-empty. The case that
    $x\not\in \ran(s_m)$ is dual.

    So, if $V\cap Z\not=\varnothing$ for some $V\in E$, then $W\cap Z
    \not=\varnothing$ for all $W\in E$. Hence since 
    elements of $E$ are pairwise disjoint it follows that
    $$\alpha < |X| = |E| \leq |\bigcup_{V\in E} V \cap Z| \leq |Z|\leq \alpha$$
    a contradiction. Hence $V\cap Z =\varnothing$, or
    equivalently, 
    \[
      V \subseteq \bigcap_{m \geq 1} \dom(s_m) \cap \ran(s_m),
    \]
    for all $V\in E$. Thus if $m \geq 1$ then $s_m|_U: U \to U$ is surjective,
    and since every element of $I(X, \alpha)$ is injective, $s_m$ is a
    permutation on $U$. Hence it follows that $Y'\setminus \dom(t_n) = X
    \setminus \dom(s_n)$ for all $n\in \N$.  In particular, $t_n\in I(Y',
    \alpha)$ for all $n\in \N$, as required.
  \end{proof}

  Since $w_1, w_2, \ldots \in A ^ +$ is universal for $I(Y, \alpha)$ and $|Y| =
  |Y'|$, it follows that  $w_1, w_2, \ldots$ is universal for $I(Y',
  \alpha)$ also.  Thus there is a homomorphism $\Phi: A ^ + \to I(Y',  \alpha)$
  such that $(w_n)\Phi = t_n$ for all $n\in \N$. We define 
  $X' = \set{(x)f}{x\in X,\ f\in (A^ +)\Phi}\cup X \subseteq Y'$.
  Since $(A ^ +)\Phi$ is countable and $|X| > \aleph_0$, it follows that
  $|X'| = |X|$. 

  Let $T$ be the inverse subsemigroup of $I(Y', \alpha)$ generated by $\{t_1,
  t_2, \ldots\}$ and let $\Omega'$ be the set of blocks of $T$ acting on
  $X'\setminus X$.  Since $|E| = |X|$ and $|\Omega'| \leq |X|$,
  there exists a
  bijection $b: E \to \Omega' \cup E$.  We will show that for every $V\in E$ there
  exists a bijection $\phi_V: V\to (V)b$ such that $t_n \phi_V= \phi_V t_n$ for all $n\in
  \N$. If $(V)b\in E$, then this follows immediately from the definition of $E$
  and since $t_n|_X=s_n$. Suppose that $(V)b\in \Omega'$. If $(x,y) \in (V)b
  \subseteq X'\setminus X \subseteq Y'\setminus X = Y \times U$, then
  $$(V)b = \set{(x, (y)s)}{s\in S} = \{x\}\times U$$
  since $U$ is a block of the action of $S$ on $X$.
  Since $U, V\in E$, there exists bijection $\phi: V \to U$ such that $\phi s_n
  = s_n \phi$ for all $n\in \N$. Define $\phi_V: V \to \{x\}\times U$ 
  so that $(a)\phi_V = (x, (a)\phi)$. Since $\phi$ is a bijection, so too is
  $\phi_V$. If $n\in \N$ and $a\in V$ are arbitrary, then 
  $$(a)\phi_Vt_n = (x, (a)\phi)t_n = (x, (a)\phi s_n) = (x, (a)s_n \phi) =
  (a)s_n\phi_V = (a)t_n\phi_V.$$

  We define $\psi: X\to X'$  by $$\psi = \bigcup_{V\in E} \phi_V \cup
  1_{X\setminus \bigcup_{W\in E} W}.$$ Note that $\psi$ is injective,
  $\dom(\psi) = X$, and $\ran(\psi) = \left( \bigcup_{W \in E} (W)b \right)
  \cup \left(X \setminus \bigcup_{W \in E} W \right) = X'$. The last equality
  holds since $b$ is a bijection from $E$ to $\Omega' \cup E$ and so by
  definition of $\Omega'$
  \[
    \bigcup_{W \in E} (W)b = \bigcup_{A \in \Omega' \cup E} A = \left( X'
    \setminus X \right) \cup B,
  \]
  where $B = \bigcup_{A \in E} A \subseteq X$.
  Hence $\psi$ is a
  bijection. We will show that $\psi t_n = s_n \psi$ for all $n\in \N$.
  Suppose that $x\in X$. Then either $x\not \in V$ for all $V\in E$ or $x\in V$
  for some $V\in E$.  In the first case, $(x)\psi t_n = (x)t_n = (x)s_n$ and
  since $(x)s_n \not\in V$ for all $V\in E$, it follows that $(x)\psi t_n=
  (x)s_n = (x)s_n\psi$, as required. In the second case, $(x)\psi t_n =
  (x)\phi_Vt_n = (x)t_n\phi_V = (x)s_n\phi_V$, and since $(x)s_n\in V$,
  $(x)s_n\phi_V= (x)s_n\psi$.

  Define $\Lambda: A ^ + \to I(X, \alpha)$ by $(w)\Lambda = \psi (w)\Phi|_{X'}
  \psi ^ {-1}$ for all $w\in A^+$. By the definition of $X'$, the partial
  permutation $(w)\Phi$ maps $X'$ to $X'$, and so $(w)\Lambda$ is a partial
  permutation of $X$. Also $$|X \setminus \dom((w)\Lambda)| = | X' \setminus
  \dom((w)\Phi)| \leq |Y'\setminus \dom((w)\Phi)| \leq \alpha$$ and similarly
  $|X \setminus \ran((w)\Lambda)| \leq \alpha$. Hence $(w)\Lambda \in I(X,
  \alpha)$. Finally, let $u, v \in A^+$. Then $$(uv)\Lambda = \psi
  (u)\Phi|_{X'} 1_{X'} (v)\Phi|_{X'} \psi^{-1} = \psi (u)\Phi|_{X'}\psi^{-1} \psi
  (v)\Phi|_{X'} \psi^{-1} = \Lambda(u) \Lambda(v),$$ and so $\Lambda$ is a
  homomorphism. Furthermore,
  $$(w_n) \Lambda = \psi\ (w_n)\Phi\ \psi ^ {-1} =  \psi
  t_n \psi ^ {-1} = s_n$$
  and hence $w_n$ is universal for $I(X, \alpha)$. 
\end{proof}

\begin{cor}\label{cor-universal-cardinlity-sym}
  Let $X$ and $Y$ be infinite sets such that $|X| < |Y|$. Then the following
  hold:
  \begin{enumerate}[\rm (i)]
    \item 
      every sequence that is universal for $\sym(X)$ is universal for
      $\sym(Y)$;
    
    \item 
      if $2^{\aleph_0} < |X|$, then every sequence that is universal for
      $\sym(Y)$ is universal for $\sym(X)$.
  \end{enumerate}
  In particular, if $2^{\aleph_0} < |X| \leq |Y|$, then the universal sequences
  for $\sym(X)$ coincide with those for $\sym(Y)$.
\end{cor}
\begin{proof}
  Part (i) follows immediately from
  Theorem~\ref{thm-different-cardinalities}(i), when $\alpha = 0$. 

  For part (ii), it suffices to show that the regularity condition in part (ii)
  of Theorem~\ref{thm-different-cardinalities} can be removed. 
  Let $w_1, w_2, \ldots$ be a universal sequence for $\sym(Y)$, 
  let $\lambda$ denote the successor cardinal of $2 ^ {\aleph_0}$, and let $Z$
  be any set of cardinality $\lambda$. Then $\lambda$
  is a regular cardinal, and so Theorem~\ref{thm-different-cardinalities}(ii)
  implies that $w_1, w_2, \ldots$ is universal for $\sym(Z)$. Therefore
  since $|X| \geq \lambda = |Z|$, it follows from part (i) that $w_1, w_2,
  \ldots$ is universal for $\sym(X)$. 
\end{proof}

The proof of the next corollary is analogous to that of
Corollary~\ref{cor-universal-cardinlity-sym}, if $\alpha = |Y|$ and we observe
that $I(X, \alpha) = I(X)$ and $I(Y, \alpha) = I(Y)$.

\begin{cor}\label{cor-universal-cardinlity-sym-inv}
  Let $X$ and $Y$ be infinite sets such that $|X| < |Y|$. Then the following
  hold:
  \begin{enumerate}[\rm (i)]
    \item 
      every sequence that is universal for $I(X)$
      is universal for $I(Y)$;
    
    \item 
      if $2^{\aleph_0} < |X|$, then every sequence that is universal 
      for $I(Y)$ is universal for $I(X)$.
  \end{enumerate}
  In particular, if $2^{\aleph_0} < |X| \leq |Y|$, then the universal sequences
  for $I(X)$ coincide with those for $I(Y)$.
\end{cor}

\begin{question}\label{question:remove-ass}
  Can the assumption that $|X| > 2^{\aleph_0}$ be removed from
  Theorem~\ref{thm-different-cardinalities}(ii) and the corollaries following
  it? 
\end{question}


\section{A sufficient condition for the universality of sequences for
$X^X$}\label{sect-trans}

In this section, we give a sufficient condition for a sequence over a 2-letter
alphabet to be universal for $X ^ X$ for any infinite $X$.  This might be seen
as a small step towards obtaining a description of the set of all universal
sequences for $X ^ X$, if such a description exists; and towards resolving the
following open question, which was the original motivation behind the results
in this section.

\begin{question}
  Let $X$ and $Y$ be infinite sets.  Is the set of universal sequences for
  $X^X$ equal to the set of universal sequences for $Y^Y$?
\end{question}

Throughout this section, we denote by $A$ a fixed alphabet $\{a, b\}$.
Let $\mathbf{w} = (w_1, w_2, \ldots)$ be a sequence of elements of $A^+$, and
let $S$ be a submonoid of $A^*$ such that:
  \begin{enumerate}

    \item 
      \label{equation-def-s-ws-0} 
      if $w_n = s v u v s'$ where $s, s' \in S$, and $u, v \in A^*$,  then
      $v \in S$;

    \item
      \label{equation-def-s-ws-1} 
      if $w_m = svt$ and $w_n = t'vs'$, $m\not =n$, where $s, s' \in S$ and
      $t, t', v \in A^*$, then  $v \in S$;

  \end{enumerate}
  where $m, n\in \N$.  For every sequence $\mathbf{w}$ of elements of $A^+$
  there is at least one submonoid of $A^ *$ satisfying these conditions, namely
  $A^*$ itself.

We will show that for every sequence $\mathbf{w}$ in $A ^ +$ there exists a
least submonoid of $A ^ *$ with respect to containment satisfying
\eqref{equation-def-s-ws-0} and \eqref{equation-def-s-ws-1}. It can be shown
that an arbitrary intersection of submonoids satisfying these three conditions,
also satisfies the conditions. However, we opt instead to give a construction
of this least submonoid, which we will make use of later.

We define $S_0 = \{ \varepsilon \}$ where $\varepsilon$ denotes the
empty word, which is the identity element of $A ^ *$.  For some $n\geq 0$, 
suppose that we have defined a submonoid $S_n$ of $A^*$. Define 
\begin{eqnarray*}
  X_n & = & \{ v \in A ^ *: w_i = s v u v s' \text{ for some } i \in \N, s, s'
  \in S_n \text{ and } u \in A^* \}; \\
  Y_n & = & \{v \in A ^ *:   w_i = s v t, w_j = t' v s' \text{ for some distinct }
  i, j \in \N, s, s' \in S_n \text{ and }  t, t' \in A^* \}. 
\end{eqnarray*}
and set $S_{n + 1} = \langle S_n, X_n, Y_n \rangle$.
We define $S_{\mathbf{w}} = \bigcup_{n \in \N} S_n$. Since 
$S_0 \leq S_1 \leq S_2 \leq \ldots$ by definition, 
$S_{\mathbf{w}}$ is a submonoid of $A^*$.

The next proposition is a straightforward consequence of the construction of
$S_{\mathbf{w}}$.

\begin{prop}
  Let $\mathbf{w} = (w_1, w_2, \ldots)$ be an arbitrary sequence of elements of
  $A^+$. Then $S_{\mathbf{w}}$ is the least submonoid of $A^*$ satisfying
  conditions \eqref{equation-def-s-ws-0} and \eqref{equation-def-s-ws-1}.
\end{prop}

The main result of this section is the following.

\begin{thm}\label{thm-universal-sequences}
  Let $\mathbf{w} = (w_1, w_2, \ldots)$ be a sequence of words in $A ^ +$ such
  that there are no $s, t, v \in A^*$ such that $w_n = stv$ with $st, tv \in
  S_{\mathbf{w}}$ for all $n \in \N$. Let $p_n, s_n, u_n \in A^*$ be such that
  $w_n = p_n u_n s_n$, and $p_n$ and $s_n$ are respectively the longest prefix
  and the longest suffix of $w_n$ so that $p_n, s_n \in S_{\mathbf{w}}$.
  Suppose that $u_n$ is a subword of $w_m$ if and only if $n = m$ and that
  $u_n$ is not a subword of $p_n$ for all $n$. Then $(w_1, w_2, \ldots)$ is a
  universal sequence for $X^X$, where $X$ is any infinite set.
\end{thm}
  
We note that the assumption on the sequence $\mathbf{w}$ in the above theorem
implies that $w_n \notin S_{\mathbf{w}}$ for all $n \in \N$. As a corollary to
Theorem~\ref{thm-universal-sequences} we obtain the following result.

\begin{cor}\label{cor-no-pre-suf-overlap}
  Let $X$ be an infinite set and let $w_1, w_2, \ldots \in A ^ +$ be such that
  no proper prefix of $w_n$ is a suffix of any $w_m$, and  $w_n$ is not a
  subword of $w_m$, $m\not = n$.  Then $(w_1, w_2, \ldots)$ is a universal
  sequence for $X^X$.
\end{cor}
\begin{proof}
  It follows from the construction of $S_{\mathbf{w}}$ where $\mathbf{w} =
  (w_1, w_2, \ldots)$, that $X_0 = Y_0 = \{ \varepsilon\}$.  Hence
  $S_{\mathbf{w}} = \{ \varepsilon \}$, and so we are done by
  Theorem~\ref{thm-universal-sequences}. 
\end{proof}

  Two examples of sequences satisfying the hypothesis of
  Corollary~\ref{cor-no-pre-suf-overlap} are $(aba ^ {n + 1}b ^ 2)_{n\in \N}$
  and $(a ^ 2 b ^ 3 (abab ^ 3) ^ {n + 1} ab ^ 2 ab ^ 3)_{n\in \N}$ of Banach
  and Sierpi\'nski mentioned in the introduction.  There are further sequences
  satisfying the hypothesis of Theorem~\ref{thm-universal-sequences} but not
  that of Corollary~\ref{cor-no-pre-suf-overlap}. For example,  it can be shown
  that if $w_n = aba(ab)^{n+1}bab \in A^+$ for all $n \in \N$, then $(w_1, w_2,
  \ldots)$ satisfies the hypothesis of Theorem~\ref{thm-universal-sequences},
  even though $ab$ is both a prefix and a suffix. In fact, $X_0 = Y_0 =
  \{\varepsilon, ab\}$ as each word contains $a^2$ exactly once, thus no prefix
  with more than $4$ letters can be a suffix. Then $S_1 = \langle ab \rangle$
  and again for the same reason as above $X_1 = Y_1 = \{\varepsilon, ab\}$ .
  Hence $S_{\mathbf{w}} = \langle ab \rangle$ and the hypothesis of
  Theorem~\ref{thm-universal-sequences} can be easily verified.

Before presenting the proof of Theorem~\ref{thm-universal-sequences} we prove a
technical result about $S_{\mathbf{w}}$.

\begin{lem}\label{lem-universal-sequences-ab}
  Let $\mathbf{w} = (w_1, w_2, \ldots)$ be an arbitrary sequence of elements of
  $A^+$ such that $a, b \notin S_{\mathbf{w}}$. Then either $w_1, w_2, \ldots
  \in aA^*b$ and $S_{\mathbf{w}} \subseteq aA^*b \cup \{\varepsilon \}$; or
  $w_1, w_2,\ldots \in bA^*a$ and $S_{\mathbf{w}} \subseteq bA^*a \cup
  \{\varepsilon \}$.
\end{lem}

\begin{proof}
  We begin by showing that $w_n \in a A ^* b$ for all $n\in \N$ or $w_n \in b A
  ^* a$ for all $n\in \N$. Suppose that $w_m \in a A ^*$ and $w_n\in A ^* a$
  for some $m,n\in\N$. Then, by conditions~\eqref{equation-def-s-ws-0}
  and~\eqref{equation-def-s-ws-1}, $a\in S_{\mathbf{w}}$, which contradicts the
  assumption of the lemma. Hence if there exists $m\in \N$ such that $w_m \in a
  A ^*$, then $w_n \in A ^*b$ for all $n\in \N$. Similarly, if $w_m\in b A ^*$,
  then $w_n\in A ^*a$ for all $n\in \N$. Hence together these imply that  $w_n
  \in a A ^* b$ for all $n\in \N$ or $w_n \in b A ^* a$ for all $n\in \N$, as
  required. Assume without loss of generality that $w_n \in a A ^* b$ for all
  $n\in \N$.  Since $S_0 = \{ \varepsilon \}$, it suffices to show that
  $X_n\cup Y_n \subseteq  aA^*b \cup \{\varepsilon \}$ for all $n \geq 0$.
  Suppose that $n \geq 0$ is arbitrary. 

  If $x\in X_n$, then there exists $m\in \N$ such that $w_m = s x u x s'$ for
  some $s, s' \in S_n$ and $u \in A^*$. If $x \in A^* a$, then since $w_m
  \in aA^*b$ there exists $q\in A ^*$ such that $w_m = a q a s'$.  Hence $a \in
  S_{\mathbf{w}}$ by \eqref{equation-def-s-ws-0}, a contradiction. Hence $x\in
  A^* b \cup \{\varepsilon\}$, and, by symmetry, $x\in a A ^* \cup \{
  \varepsilon \}$, as required. Suppose that $y\in Y_n$. Then there exist
  distinct $m, k \in \N$ such that $w_m = s y t = aq$ and $w_k = t' y s'$ where
  $s, s' \in S_n$ and $q, t, t' \in A^*$. If $y \in A^*a$, then $w_k = q'as'$
  for some $q'\in A ^ *$ and so $a\in S_{\mathbf{w}}$ by
  \eqref{equation-def-s-ws-1}, a contradiction. Hence $y \in A ^* b \cup \{
  \varepsilon\}$ and by symmetry $y\in a A^* \cup \{ \varepsilon \}$. 
\end{proof}

\begin{lem}
  Let $\mathbf{w} = (w_1, w_2, \ldots)$ be a sequence of words in $A^+$ such
  that there are no $s, t, v \in A^*$ such that $w_n = stv$ with $st, tv \in
  S_{\mathbf{w}}$ for all $n \in \N$. Then there exsits $u_n \in A^+$ such that
  $w_n = p_n u_n s_n$ where  $p_n$ and $s_n$ are the longest prefix and suffix,
  respectively, of $w_n$ belonging to $S_{\mathbf{w}}$, for all $n\in \N$. 
\end{lem}
 \begin{proof} 
   If sum of the lengths of $s_n$ and $p_n$ is bigger or equal to the length of
   $w_n$, then there exists  $s, t, v\in A ^*$ such that $w_n = s t v$, $p_n =
   s t$, and $s_n = t v$, which is a contradiction. Otherwise there is $u_n \in
   A^+$ as required.
\end{proof}

\begin{proof}[Proof of Theorem~\ref{thm-universal-sequences}]
  First suppose that $a \in S_{\mathbf{w}}$. We consider three cases: there is
  $n\in \N$ such that $b$ does not appear in $w_n$; $b$ appears at least twice
  in at least one $w_n$; and for all $n \in \N$ the letter $b$ appears exactly
  once in $w_n$. In the first case, $w_n = a^i \in S_{\mathbf{w}}$ for some $i \geq 1$,
  a contradiction. In the second case, $w_n = a^ibuba^j$ for some $i,j \geq 0$
  and some $u \in A^*$. Then $b \in S_{\mathbf{w}}$ by
  \eqref{equation-def-s-ws-0}, and so $S_{\mathbf{w}} = A^*$, a contradiction.
  In the final case, $w_n = a^{i_n} b a^{j_n}$ for some $i_n, j_n \geq 0$ and
  all $n \in \N$. Then $b \in S_{\mathbf{w}}$ by \eqref{equation-def-s-ws-1},
  again a contradiction. Therefore $a \notin S_{\mathbf{w}}$ and the symmetric
  argument shows that $b \notin S_{\mathbf{w}}$. For the rest of the proof we
  assume that $a, b \notin S_{\mathbf{w}}$. By
  Lemma~\ref{lem-universal-sequences-ab} we may assume that $w_1, w_2, \ldots
  \in aA^*b$ and $S_{\mathbf{w}} \subseteq aA^*b \cup \{\varepsilon\}$.

  Denote by $F(A)$ the free group with $A$ being the set of generators. Let $Y$
  be any set such that $|Y| = |X|$. Since $F(A)$ is countable and $Y$ is
  infinite, we may assume that $X$ is the set of eventually constant sequences
  over $F(A) \cup Y$ such that the first element is in $F(A)$. For convenience
  write the sequences from right to left, namely
  \begin{align*}
    X = \{(\ldots, x_1, x_0) : \ & x_0 \in F(A), \,  x_i \in F(A) \cup Y \,
    \text{for} \, i \geq 1, \, \text{and there is} \, K \in \N \\
    &\text{such that} \, x_K = x_k \, \text{for all} \, k \geq K\} .
  \end{align*}
  We proceed by proving a series of claims.

  \begin{claim}\label{claim-seq--1}
    $u_n \in a A^* b$ for all $n \in \N$.
  \end{claim}
  \begin{proof}
  Let $n,m \in \N$ be distinct.  Suppose that $u_n \in b A^*$. Then $u_n = b u$
  for some $u \in A^*$, thus $w_n = p_n b u s_n$. Since $w_m \in a A^* b$ there
  is some $v \in A^*$ such that $w_m = a v b$, and so condition
  \eqref{equation-def-s-ws-1} implies that $b \in S_{\mathbf{w}}$, a
  contradiction. Hence $u_n \in aA^*$ and by symmetry $u_n \in A^*b$.
  \end{proof}

  By construction $S_{\mathbf{w}}$ is generated by $G = \bigcup_{n \geq 0} X_n
  \cup Y_n$, a set of subwords of words in $\mathbf{w}$. Let $G_n$ be
  the set of all words in $G$ of length at most $n$. Recall that we say that a
  generating set $T$ is irredundant if $v$ is not an element of the monoid
  generated by $T \setminus \{v\}$ for every $v \in T$. Let $T_0 = T_1 = G_1 =
  \{ \varepsilon \}$. Then $T_1$ is irredundant and $T_0 \subseteq T_1 \subseteq
  G_1$. For some $n \in \N$, suppose that we defined $T_n$ such that $T_n$ is
  an irredundant generating set for the monoid generated by $G_n$ and  
  $T_{n - 1} \subseteq T_n \subseteq G_n$. Since $S_{\mathbf{w}}$
  is a submonoid of $A^*$, it follows that $xy$ cannot be a shorter word than
  any of $x$ or $y$ for all $x, y \in S_{\mathbf{w}}$. If $x \in G_{n + 1}
  \setminus G_n$ and $x \notin \langle T_n \rangle$ then $T_n \cup \{x\}$ is
  still irredundant. In fact, by above $x$ cannot be used to generate any word
  in $T_n$ as $x$ is of length $n + 1$ and every word in $T_n$ is of length at
  most $n$. Since $G_{n + 1} \setminus G_n$ is finite we can repeat this until
  an irredundant generating set $T_{n + 1}$ for the monoid generated by $G_{n +
  1}$ is obtained. By the construction  $T_n \subseteq T_{n + 1} \subseteq G_{n
  + 1}$. Therefore $T_n$ satisfying the conditions above exists for all $n \in
  \N$. Let $T = \bigcup_{n \in \N} T_n$. Then it is routine to verify that $T$
  is an irredundant generating set for $S_{\mathbf{w}}$. We note that $T$ only
  needs to be a monoid generating set, and so we may assume that $\varepsilon
  \notin T$.

  \begin{claim}\label{claim-seq-0}
    For each $v \in T$, there are $t, t' \in S_{\mathbf{w}}$ and $n, m \in \N$
    such that $tv$ is a prefix of $p_n$, and $vt'$ is a suffix of $s_m$.
  \end{claim}
  \begin{proof}
  Note that by construction, $T \subseteq \bigcup_{n \in \N} X_n \cup Y_n$.
  Suppose $v \in T \cap X_k$ for some $k \in \N$. Then $w_n = t v u v t'$ for
  some $n \in \N$, $t, t' \in S_k$, and $u \in A^*$. Hence $tv, vt' \in
  S_{\mathbf{w}}$, and so it then follows from the maximality of $p_n$ and
  $s_n$ that $tv$ is a prefix of $p_n$, and $vt'$ is a suffix of $s_n$. If $v
  \in T\cap Y_k$ for some $k \in \N$, then $w_n = t v q$ and $w_m = q' v t'$
  for some $n, m \in \N$, $q, q' \in A^*$, and $t, t' \in S_k$. Hence $tv, vt'
  \in S_{\mathbf{w}}$, and so $tv$ is a prefix of $p_n$, and $vt'$ is a suffix
  of $s_m$.
  \end{proof}

  \begin{claim}\label{claim-seq-1}
    For all $v \in T$ and all $n \in \N$, a prefix of $v$ is not a suffix of
    $u_n$, and a suffix of $v$ is not a prefix of $u_n$.
  \end{claim}
  \begin{proof}
  Let $v \in T$ and $n \in \N$ be arbitrary. By Claim~\ref{claim-seq-0} there
  are $t, t' \in S_{\mathbf{w}}$ such that $tv$ is a prefix of $p_m$ and $vt'$
  is a suffix of $s_k$ for some $m, k \in \N$. Then there is $r \in A^*$ so
  that $w_m = t v r u_m s_m$. Suppose that $q$ is a  non-trivial prefix of $v$
  which is also a suffix of $u_n$. First, consider the case where $m = n$. Then
  $q \in S_{\mathbf{w}}$ by \eqref{equation-def-s-ws-0} as $w_m = t q h q s_m$
  for some $h \in A^*$. If $m \neq n$, then, since $w_m = t v r u_m s_m$ and
  $w_n = p_n u_n s_n$ where $t, s_n \in S_{\mathbf{w}}$, it follows from
  \eqref{equation-def-s-ws-1} that $q \in S_{\mathbf{w}}$. Hence in both cases
  $q \in S_{\mathbf{w}}$, which contradicts the maximality of $s_n$.

  The case where $q$ is non-trivial suffix of $v$ which is a prefix of $u_n$
  follows in an almost identical way, using $w_k = p_k u_k r' v t'$ for some
  $r' \in A^*$. 
  \end{proof}

  \begin{claim}\label{claim-seq-2}
    For every $v, v' \in T$, if a non-trivial prefix $q$ of $v$ is a suffix of
    $v'$, then $q = v = v'$.
  \end{claim}
  \begin{proof}
  Let $v, v' \in T$ be arbitrary. Suppose that $v = q r$ and $v' = r' q$ for
  some $r, r' \in A^*$ and $q \in A^+$. By Claim~\ref{claim-seq-0} there are
  $t, t' \in S_{\mathbf{w}}$ and $n, m \in \N$ such that $tv$ is a prefix of
  $p_n$, and $v't'$ is a suffix of $s_m$. If $n = m$ then there is $x \in A^*$
  such that $w_n = t v x v' t' = t q r x r' q t'$, and so $q \in
  S_{\mathbf{w}}$ by \eqref{equation-def-s-ws-0} since $t, t' \in
  S_{\mathbf{w}}$. If $n \neq m$, then $w_n = t v x = t q r x$ and $w_m = x' v'
  t' = x' r' q t'$ for some $x, x' \in A^*$. Since $t, t' \in S_{\mathbf{w}}$,
  \eqref{equation-def-s-ws-1} implies that $q \in S_{\mathbf{w}}$. Hence  $q
  \in S_{\mathbf{w}}$ in both cases.

  Since $v \in T$, by Claim~\ref{claim-seq-0} there are $n, m \in \N$, $l, l'
  \in S_{\mathbf{w}}$ so that $l v$ is a prefix of $p_{n}$ and $v l'$ is a
  suffix of $s_{m}$. As in the previous paragraph, if $n = m$ then there is $x
  \in A^*$ such that $w_n = l v x v l' = l q r x q r l'$, and so $r \in
  S_{\mathbf{w}}$ by \eqref{equation-def-s-ws-0} since $lq, l' \in
  S_{\mathbf{w}}$. If $n \neq m$, then $w_n = l v x = l q r x$ and $w_m = x' v
  l' = x' q r l'$ for some $x, x' \in A^*$. Since $lq, l' \in S_{\mathbf{w}}$,
  \eqref{equation-def-s-ws-1} implies that $r \in S_{\mathbf{w}}$. Hence  $r
  \in S_{\mathbf{w}}$ in both cases. Since $T$ is irredundant, $q, r \in
  S_{\mathbf{w}}$, and $qr \in T$, it follows that $r = \varepsilon$. The same
  argument for $v'$ implies that $r' = \varepsilon$, and so $q = v = v'$.
  \end{proof}
  
  Let $f_1, f_2, \ldots \in X^X$.
  We will construct a homomorphism $\Phi : A^+ \rightarrow
  X^X$ such that $(w_n)\Phi = f_n$ for all $n \in \N$. In order to do
  that we will require the following auxiliary functions $\alpha, \beta, \gamma
  \in X ^ X$ defined as follows:
  \[
    (\ldots, x_1, x_0)\alpha  = (\ldots, x_0, a) \qquad
    \text{and}\qquad (\ldots, x_1, x_0)\beta  = (\ldots, x_0, b).
  \]
  If $x_{i-1}\ldots x_0 = v \in T$ for some $i \geq 1$, $x_j \in A^+$ for all
  $j \in \{0, \ldots i -1\}$, and $x_i \in F(A)$, we define 
  \[
    (\ldots, x_1, x_0)\gamma = (\ldots, x_{i+1}, x_i v)
  \]
  and otherwise define $(\ldots, x_1, x_0)\gamma = (\ldots, x_1, x_0)$.

  Suppose there are $i, i' \in \N$, such that $i \geq i'$,  $x_{i-1}\ldots x_0
  = v$, and $x_{i' - 1} \ldots x_0 = v'$ for some $v, v' \in T$, and so that
  $x_j \in A^+$ for all $j \in \{0, \ldots, i - 1\}$. Then $v'$ is a suffix of
  $v$. By Claim~\ref{claim-seq-2} this is only possible if $v = v'$. Hence
  $\gamma$ is well-defined. Let $\Psi : A^+ \to X^X$ be the canonical
  homomorphism induced by $(a)\Psi = \alpha$ and $(b)\Psi = \beta \circ
  \gamma$. We will later use $\Psi$ to define the required $\Phi$.

  \begin{claim}\label{claim-seq-3}
    For $v \in aA^*$ such that no prefix of $v$ is a suffix of a word in $T$,
    there are $z_1, \ldots, z_k \in A^+$ such that $z_1\ldots z_k = v$ and $
    (\ldots, x_1, x_0) \left( (v) \Psi \right) = (\ldots, x_1, x_0,z_1, \ldots,
    z_k) $ for every $(\ldots, x_1, x_0) \in X$.
  \end{claim}

  \begin{proof}
  Let $v \in aA^*$ be such that no prefix of $v$ is a suffix of a word in $T$,
  and let $v = y_1 \ldots y_m$ for some $m \in \N$ and $y_1, \ldots, y_m \in
  A$. Then $y_1 = a$, and so $ (\ldots, x_1, x_0) \alpha = (\ldots, x_1, x_0,
  y_1)$ for all $(\ldots, x_1, x_0) \in X$. Suppose that for some $i \in \{1,
  \ldots, m - 1\}$ there are $j \in \N$ and $z_1, \ldots, z_j \in A^+$ such that
  $ (\ldots, x_1, x_0) \left( (y_1 \ldots y_i )\Psi \right) = (\ldots, x_1,
  x_0, z_1, \ldots, z_j)$ for every $(\ldots, x_1, x_0) \in X$ and $y_1\ldots
  y_i = z_1 \ldots z_j$. We proceed with an induction on $i$.

  In order to prove the inductive step, there are two cases to consider, either
  $y_{i + 1} = a$, or $y_{i + 1} = b$. Suppose that $y_{i+1} = a$. Since $\Psi$
  is a homomorphism, $(\ldots, x_1, x_0) \left( (y_1 \ldots y_{i+1})\Psi
  \right) = (\ldots, x_1, x_0, z_1, \ldots ,z_j, a)$ for all $(\ldots, x_1,
  x_0) \in X$
  and $z_1\ldots z_j a = y_1 \ldots y_{i + 1}$, as required.

  Suppose that $y_{i+1} = b$. Then $ (\ldots, x_1,
  x_0) \left( (y_1 \ldots y_{i+1})\Psi \right) = (\ldots, x_1, x_0, z_1, \ldots
  ,z_j, b) \gamma$ for all $(\ldots, x_1, x_0) \in X$ 
  and $z_1\ldots z_j b = y_1 \ldots y_{i + 1}$, as $\Psi$ is a homomorphism.
  Since $y_1\ldots y_{i + 1}$ is a prefix of $v$, by the assumption it cannot
  be a suffix of any word in $T$. Thus $z_1\ldots z_jb \notin T$ and if $x_0,
  \ldots, x_t \in A^+$ then $x_t \ldots x_0 z_1 \ldots z_j b \notin T$ for all
  $t \in \N$. Hence either $\gamma$ acts as the identity on $(\ldots, x_1, x_0,
  z_1, \ldots ,z_j, b)$, or there is $k > 1$ such that $z_k \ldots z_j b \in
  T$. In the later case
  \begin{align*}
    (\ldots, x_1, x_0)\left( (y_1 \ldots y_{i+1})\Psi\right)
    &= (\ldots, x_1, x_0, z_1, \ldots ,z_j, b) \gamma \\
    &= (\ldots, x_1, x_0, z_1, \ldots , z_{k - 2},z_{k - 1}z_k \ldots z_j b),
  \end{align*}
  and $z_1\ldots z_j b = y_1 \ldots y_{i + 1}$. In both cases there are $j \in
  \N$ and $z_1, \ldots, z_j \in A^+$ such that $ (\ldots, x_1, x_0) \left( (y_1
  \ldots y_{i + 1} )\Psi \right) = (\ldots, x_1, x_0, z_1, \ldots, z_j)$ for
  every $(\ldots, x_1, x_0) \in X$ and $y_1\ldots y_{i + 1} = z_1 \ldots z_j$,
  which proves  the inductive step. Hence the claim holds by induction.
  \end{proof}

  \begin{claim}\label{claim-seq-4}
    Let $v \in S_{\mathbf{w}}$. Then  $(\ldots, x_1, x_0) \left((v)\Psi \right)
    = (\ldots, x_1, x_0v)$ for all $(\ldots, x_1, x_0) \in X$ and $(v)\Psi$
    is a bijection.  
  \end{claim}
  \begin{proof}
  Let $v \in T$. Then $v \in aA^*b$ as $S_{\mathbf{w}} \subseteq aA^*b \cup
  \{\varepsilon\}$, and so $v = v' b$ for some $v' \in aA^*$. By
  Claim~\ref{claim-seq-2} any proper prefix of $v$, and hence any prefix of
  $v'$, is not a suffix of any word in $T$. Hence by Claim~\ref{claim-seq-3}
  there exists $j \in \N$ and $z_1, \ldots, z_j \in A^+$ such that   $(\ldots,
  x_1, x_0)\left( (v')\Psi \right) = (\ldots, x_1, x_0, z_1, \ldots, z_j)$ for
  all $(\ldots, x_1, x_0) \in X$ and $z_1\ldots z_j = v'$. 
  Since $v = z_1\ldots z_j b$, $\Psi$ is a homomorphism, and $x_0 \in F(A)$, it
  follows that
  \begin{equation}\label{equation-psi-on-gen}
    (\ldots, x_1, x_0)\left( (v)\Psi \right) = (\ldots, x_1, x_0, z_1, \ldots,
    z_j, b) \gamma = (\ldots, x_1, x_0 v).
  \end{equation}

  Clearly, $(\ldots, x_1, x_0) \mapsto (\ldots, x_1, x_0 v^{-1})$ is the
  inverse map of $(v)\Psi$. Therefore, we are done, as $T$ is a generating set
  for $S_{\mathbf{w}}$. 
  \end{proof}

  In order to define the required $\Phi$, we need a final auxiliary 
  function $\delta \in X^X$, defined as follows. 
  If there exist $n, i \geq 1$, such that $x_{i-1} \cdots x_0 = u_n$,
  $x_0, \ldots, x_{i - 1} \in A^+$, and $x_i \in F(A)$, then we define
  \[ (\ldots, x_1, x_0)\delta=
          (\ldots, x_{i+1}, x_i p_n^{-1})f_n \circ ((s_n)\Psi)^{-1} 
  \] 
  and we define $(\ldots, x_1, x_0)\delta= (\ldots, x_1, x_0)$ otherwise.
  Note that $((s_n)\Psi)^{-1}$ is defined by Claim~\ref{claim-seq-4}. Suppose
  there are $i, i', n, n' \in \N$, $i \geq i'$ such that $x_{i-1}\ldots x_0 =
  u_n$ and $x_{i' - 1} \ldots x_0 = u_{n'}$ where $x_j \in A^+$ for all $j \in
  \{0, \ldots, i - 1\}$ and $x_i, x_{i'} \in F(A)$. Then $u_{n'}$ is a
  suffix of $u_n$.  On the other hand, if $n'\not= n$, then  $u_{n'}$ is not a
  subword of $w_n$ (by assumption in the statement of the theorem) and hence
  not of $u_n$ either.  Hence $n = n'$, and so $i = i'$, and $\delta$ is
  well-defined. 

  Let $\Phi$ be the canonical homomorphism
  induced by $(a)\Phi = \alpha$ and $(b)\Phi = \beta \circ \gamma \circ
  \delta$.

  \begin{claim}\label{claim-seq-5}
    If $v \in S_{\mathbf{w}}$, then $(v)\Phi = (v)\Psi$.
  \end{claim}
  \begin{proof}
  Suppose that $v = y_1 \ldots y_m \in T$ where $y_i \in A$ for all $i \in \{1,
  \ldots, m\}$. Since $S_{\mathbf{w}} \subseteq a A^* b \cup \{\varepsilon \}$,
  it follows that $y_1 = a$, and so $(y_1)\Phi = \alpha = (y_1)\Psi$. Suppose
  $(y_1 \ldots y_i) \Phi  = (y_1 \ldots y_i) \Psi $ for some $i \in \{1,
  \ldots, m - 1\}$. We proceed by indution on $i$. 
  
  It follows from the inductive hypothesis that $(y_1 \ldots y_{i+1}) \Phi  =
  (y_1 \ldots y_i) \Psi \circ (y_{i + 1})\Phi$.
  If $y_{i + 1} = a$, then $(y_{i + 1}) \Phi = (y_{i + 1}) \Psi$, proving the
  first case of the inductive step. Suppose that $y_{i + 1} = b$, then $(y_{i
  + 1}) \Phi =(y_{i + 1}) \Psi \circ \delta$, and so $(y_1 \ldots y_{i+1}) \Phi
  = (y_1 \ldots y_{i + 1}) \Psi \circ \delta$. If $i + 1 < m$, then $y_1 \ldots
  y_{i + 1}$ is a proper prefix of $v$. By Claim~\ref{claim-seq-2} for any $j
  \in \{1, \ldots, i + 1\}$ the proper prefix $y_1 \ldots y_j$ of $v$ is a not
  a suffix of any word in $T$. Since $y_1 \ldots y_{i + 1} \in aA^*$, by
  Claim~\ref{claim-seq-3} there exists $j \in \N$ and $z_1, \ldots, z_j \in
  A^+$ such that $z_1\ldots z_j = y_1 \ldots y_{i + 1}$ and  $(\ldots, x_1,
  x_0) \left( (y_1 \ldots y_{i + 1})\Psi \right) = (\ldots, x_1, x_0, z_1,
  \ldots, z_j)$ for all $(\ldots, x_1, x_0) \in X$. If $i + 1 = m$, then $y_1
  \ldots y_{i + 1} = v \in S_{\mathbf{w}}$, and so $(\ldots, x_1, x_0) \left(
  (y_1 \ldots y_{i + 1})\Psi \right) = (\ldots, x_1, x_0 y_1 \ldots y_{i + 1})$
  for all $(\ldots, x_1, x_0) \in X$ by Claim~\ref{claim-seq-4}. Hence in any 
  case there are $j \geq 0$, $z_0 \in A^*$, and $z_1, \ldots, z_j \in A^+$
  such that $z_0\ldots z_j = y_1 \ldots y_{i + 1}$ and for all $(\ldots, x_1,
  x_0) \in X$
  \begin{equation}\label{equation-step-in-claim-8}
    (\ldots, x_1, x_0) \left( (y_1 \ldots y_{i + 1})\Psi \right)
    = (\ldots, x_1, x_0z_0, z_1, \ldots, z_j).
  \end{equation}

  We will show that $\delta$ acts as the identity on $(\ldots, x_1, x_0)
  \left((y_1 \ldots y_{i + 1})\Psi \right)$ for all $(\ldots, x_1, x_0) \in X$.
  Fix $(\ldots, x_1, x_0) \in X$, and let $z_0, \ldots, z_j \in A^+$ be as in
  \eqref{equation-step-in-claim-8}. Suppose that there are $k, n \geq 1$
  such that $x_{k - 1}, \ldots, x_1, x_0z_0 \in A^+$, $x_k \in F(A)$,
  and $x_{k - 1} \ldots x_0 z_0 \ldots z_j = u_n$. Then $z_0 \ldots z_j = y_1
  \ldots y_{i + 1}$ is both a prefix of $v$ and a suffix of $u_n$,
  contradicting Claim~\ref{claim-seq-1}. If $k > 0$ and $z_k \ldots z_j = u_n$,
  then $u_n$ is a subword of $v$ for some $n \in \N$. By
  Claim~\ref{claim-seq-0} there are $t \in S_{\mathbf{w}}$ and $m \in \N$ such
  that $tv$ is a prefix of $p_m$, and so $u_n$ is a subword of $p_m$. Moreover,
  by the assumption of the theorem $m = n$, contradicting by the hypothesis of
  the theorem.  Hence $\delta$ acts as identity on $(\ldots, x_1, x_0z_0, z_1,
  \ldots, z_j)$, proving that $(y_1 \ldots y_{i + 1}) \Phi  = (y_1 \ldots y_{i
  + 1}) \Psi$, and so the inductive step. It then follows by induction that
  $(y_1 \ldots y_i)\Phi = (y_1 \ldots y_i)\Psi$ for all $i \in \{1, \ldots,
  m\}$. In particular, if $i = m$, then $(v)\Phi = (v)\Psi$. Since $v \in T$ is
  arbitrary and $T$ is a generating set for $S_{\mathbf{w}}$, it follows that
  $(v) \Phi = (v)\Psi$ for all $v \in S_{\mathbf{w}}$.
  \end{proof}

  \begin{claim}\label{claim-seq-6}
    $(u_n)\Phi = (u_n) \Psi \circ \delta$ for all $n \in \N$.
  \end{claim}
  \begin{proof}
  Let $n \in \N$, and let $u_n = y_1 \ldots y_m$ where $y_1, \ldots, y_m \in
  A$. We will now show that $(y_1\ldots y_{m - 1})\Phi = (y_1 \ldots y_{m -
  1})\Psi$. Since $y_1 = a$ by Claim~\ref{claim-seq--1}, it follows that $(y_1)\Phi = \alpha = (y_1)\Psi$.
  Suppose $(y_1 \ldots y_i) \Phi = (y_1 \ldots y_i) \Psi $ for some $i \in \{1,
  \ldots, m - 2\}$. Then $ (y_1 \ldots y_{i+1}) \Phi  =  (y_1 \ldots y_i) \Psi
  \circ (y_{i + 1})\Phi$. If $y_{i + 1} = a$, then $(y_{i + 1}) \Phi = (y_{i +
  1}) \Psi$, and so the inductive hypothesis is satisfied. Suppose $y_{i + 1} =
  b$. Then $(y_{i + 1}) \Phi = (y_{i + 1}) \Psi \circ \delta$. Hence $(y_1
  \ldots y_{i+1}) \Phi = (y_1 \ldots y_{i + 1}) \Psi \circ \delta$. By
  Claim~\ref{claim-seq-1}, for every $j \in \{1, \ldots, i + 1\}$ the proper
  prefix $y_1 \ldots y_j$ of $u_n$ is not a suffix of any word in $T$. By
  Claim~\ref{claim-seq--1}, $y_1 \ldots y_j \in aA^*$, and so by
  Claim~\ref{claim-seq-3} there exists $j \in \N$ and $z_1, \ldots, z_j \in
  A^+$ such that $(\ldots, x_1, x_0) \left( (y_1 \ldots y_{i + 1})\Psi \right)
  = (\ldots, x_1, x_0, z_1, \ldots, z_j)$ for all $(\ldots, x_1, x_0) \in X$
  and $z_1\ldots z_j = y_1 \ldots y_{i + 1}$.

  Suppose that $z_k \ldots z_j = u_t$ for some $k \in \{1, \ldots, j\}$ and $t
  \in \N$. Then $u_t$ is a subword of $u_n$, and so of $w_n$. Hence $t = n$ by
  the hypothesis of the theorem, and thus $u_n$ is a proper subword of $u_n$,
  which is a contradiction. Suppose that $u_t = x_k \ldots x_0 z_1 \ldots z_j$
  for some $k \geq 0$ and $t \in \N$ such that $x_0, \ldots, x_k \in A^+$. 
  Then $z_1 \ldots z_j$ is a prefix of $u_n$ and a suffix of $u_t$, and so $z_1
  \ldots z_j \in S_{\mathbf{w}}$, since $S_{\mathbf{w}}$ satisfies
  condition~\eqref{equation-def-s-ws-1}. But then $w_n = p_nu_ns_n = (p_nz_1
  \ldots z_j )(y_{i+2} \ldots y_m)s_n$ with $p_n z_1 \ldots z_j \in
  S_{\mathbf{w}}$, and this contradicts the maximality of the length of $p_n$.
  So $\delta$ acts as the identity on $(\ldots, x_1, x_0, z_1, \ldots, z_j)$.
  Hence $(y_1 \ldots y_{i + 1})\Phi = (y_1 \ldots y_{i + 1})\Psi$. By induction
  $(y_1 \ldots y_{m - 1})\Phi = (y_1 \ldots y_{m - 1})\Psi$. Finally,
  $(u_n)\Phi = (u_n)\Psi \circ \delta$, as $y_m = b$.
  \end{proof}

  Let $n \in \N$. It follows from Claim~\ref{claim-seq-4}, Claim~\ref{claim-seq-5},
  Claims~\ref{claim-seq-6}, and the fact that $\Phi$ is a homomorphism, that for
  all $(\ldots, x_1, x_0) \in X$
  \begin{align*}
    (\ldots, x_1, x_0)(w_n) \Phi &= (\ldots, x_1, x_0) \left( (p_n) \Psi \circ (u_n)
      \Psi \circ \delta \circ (s_n)\Psi \right) \\
                     &= (\ldots, x_1, x_0p_n) \left((u_n) \Psi \circ
      \delta \circ (s_n)\Psi \right).
  \end{align*}
  It follows from Claims~\ref{claim-seq--1},~\ref{claim-seq-1}
  and~\ref{claim-seq-3} that there are $z_1, \ldots, z_k \in A^+$ such that
  $z_1 \ldots z_k = u_n$ and
  \begin{align*}
    (\ldots, x_1, x_0)(w_n) \Phi &= (\ldots, x_1, x_0p_n) \left((u_n) \Psi \circ
      \delta \circ (s_n)\Psi \right) \\
                     &= (\ldots, x_1, x_0p_n,z_1, z_2, \ldots, z_k)
      \delta \circ (s_n)\Psi. \\
  \end{align*}
  Finally, by the definition of $\delta$
  \begin{align*}
    (\ldots, x_1, x_0)(w_n) \Phi &= (\ldots, x_1, x_0p_n,z_1, z_2, \ldots, z_k)
      \delta \circ (s_n)\Psi \\
                   &= (\ldots, x_1, x_0)f_n \circ ((s_n)\Psi)^{-1} \circ (s_n) \Psi\\
                   &= (\ldots, x_1, x_0)f_n.
  \end{align*}
  Therefore $(w_n)\Phi = f_n$, and since $n$ was arbitrary, $(w_1, w_2, \ldots)$
  is a universal sequence.
\end{proof}

\subsection*{Acknowledgements}
  The authors would like to thank Manfred Droste for pointing out that the
  condition that $|X|$ is a regular cardinal was not required in
  Corollaries~\ref{cor-universal-cardinlity-sym}. The authors also 
  thank the anonymous referee for their helpful comments and corrections.


\bibliography{universal}{}

\begin{thebibliography}{10}

\bibitem{Banach1935aa}
S.~Banach.
\newblock Sur un theor\`eme de m. sierpi\'nski.
\newblock {\em Fund. Math.}, 25:5--6, 1935.

\bibitem{Bergman2006aa}
G.~M. Bergman.
\newblock Generating infinite symmetric group.
\newblock {\em Bull. London Math. Soc.}, 38:429--440, 2006.

\bibitem{:2006fu}
G.~M. Bergman.
\newblock Problem list from {\it {a}lgebras, lattices and varieties}: a
  conference in honor of {W}alter {T}aylor, {U}niversity of {C}olorado, 15--18
  {A}ugust, 2004.
\newblock {\em Algebra Universalis}, 55(4):509--526, 2006.

\bibitem{Dougherty1999aa}
Randall Dougherty and Jan Mycielski.
\newblock Representations of infinite permutations by words. {II}.
\newblock {\em Proc. Amer. Math. Soc.}, 127(8):2233--2243, 1999.

\bibitem{Droste2006aa}
M.~Droste and J.~K. Truss.
\newblock On representing words in the automorphism group of the random graph.
\newblock {\em J. Group Theory}, 9(6):815--836, 2006.

\bibitem{Droste1985ac}
Manfred Droste.
\newblock Classes of universal words for the infinite symmetric groups.
\newblock {\em Algebra Universalis}, 20(2):205--216, 1985.

\bibitem{Droste1987aa}
Manfred Droste and Saharon Shelah.
\newblock On the universality of systems of words in permutation groups.
\newblock {\em Pacific J. Math.}, 127(2):321--328, 1987.

\bibitem{East2012aa}
James East.
\newblock Generation of infinite factorizable inverse monoids.
\newblock {\em Semigroup Forum}, 84(2):267--283, 2012.

\bibitem{East2014aa}
James East.
\newblock Infinite partition monoids.
\newblock {\em Internat. J. Algebra Comput.}, 24(4):429--460, 2014.

\bibitem{East2017aa}
James East.
\newblock Infinite dual symmetric inverse monoids.
\newblock {\em Periodica Mathematica Hungarica}, Jul 2017.

\bibitem{Galvin1995aa}
Fred Galvin.
\newblock Generating countable sets of permutations.
\newblock {\em J. London Math. Soc. (2)}, 51(2):230--242, 1995.

\bibitem{Howie1995aa}
John~M. Howie.
\newblock {\em Fundamentals of semigroup theory}, volume~12 of {\em London
  Mathematical Society Monographs. New Series}.
\newblock The Clarendon Press Oxford University Press, New York, 1995.
\newblock Oxford Science Publications.

\bibitem{Hyde2016aa}
J.~Hyde, J.~Jonu\v{s}as, J.~D. Mitchell, and Y.~P{\'e}resse.
\newblock Universal sequences for the order-automorphisms of the rationals.
\newblock {\em J. Lond. Math. Soc. (2)}, 94(1):21--37, 2016.

\bibitem{Kechris2007aa}
Alexander~S. Kechris and Christian Rosendal.
\newblock Turbulence, amalgamation, and generic automorphisms of homogeneous
  structures.
\newblock {\em Proc. Lond. Math. Soc. (3)}, 94(2):302--350, 2007.

\bibitem{Lyndon1990aa}
Roger~C. Lyndon.
\newblock Words and infinite permutations.
\newblock In {\em Mots}, Lang. Raison. Calc., pages 143--152. Herm\`es, Paris,
  1990.

\bibitem{Maltcev2009aa}
V.~Maltcev, J.~D. Mitchell, and N.~Ru{\v{s}}kuc.
\newblock The {B}ergman property for semigroups.
\newblock {\em J. Lond. Math. Soc. (2)}, 80(1):212--232, 2009.

\bibitem{McNulty:1976mi}
George~F. McNulty.
\newblock The decision problem for equational bases of algebras.
\newblock {\em Ann. Math. Logic}, 10(3-4):193--259, 1976.

\bibitem{Mitchell2011ab}
J.~D. Mitchell and Y.~P{{\'e}}resse.
\newblock Generating countable sets of surjective functions.
\newblock {\em Fund. Math.}, 213(1):67--93, 2011.

\bibitem{Mycielski1987aa}
Jan Mycielski.
\newblock Representations of infinite permutations by words.
\newblock {\em Proc. Amer. Math. Soc.}, 100(2):237--241, 1987.

\bibitem{Ore1951aa}
Oystein Ore.
\newblock Some remarks on commutators.
\newblock {\em Proc. Amer. Math. Soc.}, 2:307--314, 1951.

\bibitem{Sierpinski1934aa}
W.~Sierpi\'nski.
\newblock Sur l'approximation des fonctions continues par les superpositions de
  quatre fonction.
\newblock {\em Fund. Math.}, 23:119--120, 1934.

\bibitem{Sierpinski1935aa}
W.~Sierpi\'nski.
\newblock Sur les suites infinies de fonctions d\'efinies dans les ensembles
  quelconques.
\newblock {\em Fund. Math.}, 24:209--212, 1935.

\bibitem{Silberger1983aa}
D.~M. Silberger.
\newblock Are primitive words universal for infinite symmetric groups?
\newblock {\em Trans. Amer. Math. Soc.}, 276(2):841--852, 1983.

\bibitem{Taylor1981aa}
Walter Taylor.
\newblock Some universal sets of terms.
\newblock {\em Trans. Amer. Math. Soc.}, 267(2):595--607, 1981.

\bibitem{Truss2009aa}
J.~K. Truss.
\newblock private communication, 2009.

\end{thebibliography}
\bibliographystyle{plain}

\end{document}